\documentclass{amsart}
\usepackage{titlecaps}

\usepackage{ mathrsfs }
\usepackage{ amsfonts }
\usepackage{ dsfont }
\usepackage{ amsthm }
\usepackage{ amsmath }
\usepackage{ amssymb }
\usepackage{ verbatim }
\usepackage{ tikz-cd }
\usepackage{ mathtools }
\usepackage{ color }
\usepackage{hyperref}
\usepackage{url,amssymb,enumerate,colonequals}
\usepackage{xcolor}

\usepackage[symbol]{footmisc}
\usepackage{comment}

\newtheorem{thm}{Theorem}

\newtheorem{lem}[thm]{Lemma}

\newtheorem{cor}[thm]{Corollary}

\theoremstyle{definition}

\newtheorem{remark}[thm]{Remark}

\newcommand{\PP}{\mathbb{P}}
 
\newcommand{\Z}{\mathbb{Z}}

\newcommand{\FF}{\mathbb{F}}

\newcommand{\Q}{\mathbb{Q}}

\DeclareMathOperator{\Gal}{Gal}

\DeclareMathOperator{\divv}{div}

\title[Primitive points on some low degree Fermat curves]{Primitive points on some low degree\\ Fermat curves}

\author{Maleeha Khawaja}
\address{Mathematics Institute, University of Warwick, CV4 7AL, United Kingdom}
\email{Maleeha.Khawaja@warwick.ac.uk}

\date{\today}


\begin{document}

\begin{abstract} 
Let $n\geq 3$ be an integer. Let $F_n$ be the Fermat curve defined by the Fermat equation $x^n+y^n=z^n$. For a curve $C/\Q$, we say an algebraic point $P\in C(\bar{\Q})$ is primitive if the Galois group of the Galois closure of the number field $\Q(P)$ is a primitive permutation group. Recall that $A_4$ is a primitive subgroup of $S_4$. We prove that there are no non-trivial quartic points on $F_n$ with Galois closure $A_4$, when $n=7$ and $n=8$. We also provide sufficient conditions for the non-existence of non-trivial points on the Fermat curves $F_6$ and $F_8$ defined over a given primitive number field of degree at least $3$.  
\end{abstract}

\maketitle

\section{Introduction}

For an integer $n\geq 3$, let $F_{n}/\Q$ be the Fermat curve of degree $n$ defined by the Fermat equation
\begin{equation}
    \label{eq:Fermat}
    x^n+y^n=z^n.
\end{equation}
The study of low degree points on the Fermat curve of degree $n$, when $n$ is a small integer, is of long-standing interest. We begin by mentioning a few results established in this area. Aigner~\cite{Aigner34}, in 1934, proved that $\Q(\sqrt{-7})$ is the only quadratic field to possess a non-trivial solution to the Fermat quartic $F_4$ (i.e. a solution for which $xyz\neq 0$). Mordell~\cite{Mo} reproved this result in 1967, by providing an explicit parametrization of all such points. In 1978, Gross and Rohrlich \cite{Rohrlich1978} completely determined all degree $d$ points on the Fermat curve $F_n$ for $n=5, 7 $ and $11$, for all $d\leq (n-1)/2$, through studying the Mordell--Weil group of the Jacobian of $F_n$. More recently, there has been increased interest in analysing the Galois groups of cubic and quartic points on certain low degree Fermat curves. Klassen and Tzermias~\cite{MatthewKlassen1997} showed that every quartic point on the Fermat quintic $F_5$ arises as the intersection of $F_5$ with a straight line passing through a rational point. 
Using this classification, Kraus~\cite{Kraus18} gave a complete description of the quartic points on $F_5$, proving that all but finitely many such points have Galois group $D_4$. Note that as $a,b\in \Z$ vary, we obtain cubic points defined over $\Q(\sqrt[3]{a^3+b^3})$. The discriminant of the given cubic field is $-27(a^3+b^3)^2$, which implies that the Galois closure of the cubic field has Galois group $S_3$. Bremner and Choudhry~\cite{BremnerChoudhry} have provided an explicit construction of infinitely many points on the Fermat cubic with cyclic Galois group. They also proved that all cubic points on the Fermat quartic $F_4$ have Galois group $S_3$ using the fact that they arise as the intersection of $F_4$ with a straight line passing through a rational point. In this paper, we are concerned with algebraic points on the Fermat curves of degrees $6, 7$ and $8$ which have primitive Galois group. We recall the notion of a primitive Galois group for the convenience of the reader. 
Let $G$ be a group acting transitively on a set $X$. Then $G$ is said to \textbf{act primitively} on $X$ if $G$ does not stabilise a non-trivial partition of $X$. Let $K$ be a number field with Galois closure $\tilde{K}$. Recall that $K$ is a \textbf{primitive} number field if there are no proper subextensions of $K/\Q$. Suppose $G=\Gal(\tilde{K}/\Q)$. Then the Galois correspondence implies that $G$ is a primitive permutation group if and only if $K$ is a primitive number field (see e.g. \cite[Lemma 28]{khawaja2023primitive} for a proof of this equivalence). 
It is well-known that if a group acts $2$-transitively on a set then that action is primitive. As a consequence, for an integer $d\geq 3$, the groups $A_d$ and $S_d$ are primitive permutation groups. We state our main results below.

\begin{thm}\label{thm:F8quartic}
    Let $n=7$ or $8$.
    Let $K$ be a quartic field such that $\Gal(\tilde{K}/\Q)=A_4$, where $\tilde{K}$ is the Galois closure of $K$. Then there are no non-trivial solutions to \eqref{eq:Fermat} with exponent $n$ over $K$.
\end{thm}

Our main tool is the use of a morphism from the degree $n$ Fermat curve to a degree $n$ hyperelliptic curve (defined in Section~\ref{sec:hyperellipticFermat}). There are two advantages of working with the mentioned degree $n$ hyperelliptic curve. The obvious advantage is the significant reduction in the genera of the curves we are working with (the hyperelliptic curve has genus $3$ in both cases). Secondly, analysing the Riemann--Roch spaces of effective degree $4$ divisors on these curves yields an explicit parametrization of all quartic points with Galois group $A_4$. We note in passing that the existence of quartic points on the Fermat quartic with Galois group $A_4$ is currently unknown, and refer the interested reader to \cite[Section 4]{BremnerChoudhry} for a discussion of this problem.

We also use this family of hyperelliptic curves to prove sufficient conditions 
for the Fermat curves of degrees $6$ and $8$ to possess no non-trivial points 
defined over a given primitive number field of degree at least $3$.

\begin{thm}
    \label{thm:F6}
    Let $K$ be a primitive number field of degree at least $3$. 
    Suppose $E(K)=E(\Q)$, where $E/\Q$ is either of the two elliptic curves with Cremona label \texttt{27a3} or \texttt{432b1}.
    Then there are no non-trivial solutions to \eqref{eq:Fermat} for $n=6$ over $K$.
\end{thm}

\begin{thm}
    \label{thm:F8}
    Let $E$ be the elliptic curve with Cremona label \texttt{64a4}. 
    Let $K$ be a primitive number field of degree at least $3$ such that $E(K)=E(\Q)$. 
    Then there are no non-trivial solutions to \eqref{eq:Fermat} for $n=8$ over $K$.
\end{thm}

In the resolution of the Fermat equation over a fixed number field $K$ using the modular approach, it is often necessary to rule out the existence of non-trivial $K$-rational points on $F_n$ for small values of $n$ using a different argument (see e.g. \cite{JarvisMeekin} or \cite{KhawajaJarvis}). This is the intended application of these results. It would also be of interest to find families of number fields which satisfy the given criteria.\\

\textbf{Outline of the paper.} 
In Section~\ref{sec:divisors}, we recall some well-known results about divisors on curves. In Section~\ref{sec:hyperellipticFermat}, we define the family of hyperelliptic curves $C_n$ referred to above, and determine $C_n(\Q)$ for $n=6, 7$ and $8$. We prove Theorem~\ref{thm:F8quartic} for $n=7$ and $n=8$ in Sections~\ref{sec:F87} and \ref{sec:F88}, respectively. Finally, we prove Theorems~\ref{thm:F6} and \ref{thm:F8} in Sections~\ref{sec:proofF6} and \ref{sec:proofF8}, respectively. 

\section{Divisors on curves}\label{sec:divisors}

In this section, we recall some well-known results concerning divisors of algebraic points on curves. Let $C$ be a curve defined over $\Q$. 
Then a \textbf{divisor} $D$ on $C$ is a formal finite sum of algebraic points on $C$, that is,
\[
D = \sum_{P\in C(\bar{\Q})} n_{P}P,\qquad n_P\in\Z. 
\]
We say $D$ is an \textbf{effective divisor}, and write $D\geq 0$, if $n_P\geq 0$ for all $P\in C(\bar{\Q})$. The \textbf{Riemann--Roch space of} $D$, denoted by $L(D)$, is defined as
\[
L(D)=\{f\in \Q(C)^{\times}\; : \;\divv(f)+D\geq 0 \}\cup \{0\},
\]
and we denote its dimension by $\ell(D)$. 

\begin{thm}[Riemann--Roch]\label{thm:RR}
    Let $K_C$ be a canonical divisor on $C$.
    Then 
    \[
    \ell(D)-\ell(K_C-D) = \deg(D)-g+1.
    \]
\end{thm}
\begin{proof}
    See, for example, \cite{SilvermanAEC}.
\end{proof}
We say an effective divisor $D$ is \textbf{special} if $\ell(K_C-D)>0$. 
Recall that a hyperelliptic curve $C/\Q$ admits a degree $2$ morphism $\pi$ to $\PP^1$ given by $\pi(x,y)=x$, and a \textbf{hyperelliptic divisor} on $C$ is one of the form $\pi^{*}(x)$ for $x\in\PP^1$. 
\begin{thm}[Clifford]\label{thm:Clifford}
    Let $D$ be a special divisor on $C$.
    Then 
    \[
    \ell(D)\leq \frac{\deg(D)}{2} + 1
    \]
    with equality if and only if $D=0$ or $D$ is a canonical divisor or $C$ is a hyperelliptic curve and $D$ is a multiple of a hyperelliptic divisor. 
\end{thm}
\begin{proof}
    See, for example, \cite[Chapter III.1]{Hartshorne}.
\end{proof}

\begin{cor}\label{cor:degree4}
    Suppose $C/\Q$ is a genus $3$ hyperelliptic curve.
    Let $D$ be an effective degree $4$ divisor on $C$. 
    Then $\ell(D)=2$, unless $D$ is a canonical divisor or a multiple of a hyperelliptic divisor in which case $\ell(D)=3$. 
\end{cor}

\begin{proof}
    The corollary follows immediately from Theorems~\ref{thm:RR} and \ref{thm:Clifford}. 
\end{proof}

\section{Hyperelliptic Fermat curves}\label{sec:hyperellipticFermat}

The morphism given in the following lemma is well-known but we are unable to find a reference. 

\begin{lem}
    \label{lem:hyperellipticlem}
    Let $n\geq 3$ be an integer. 
    Let $C_n$ be the degree $n$ hyperelliptic curve defined as
    \[
        C_{n}: y^2 = -4x^n + 1.
    \]
    There is a non-constant morphism 
    $\pi: F_n\rightarrow C_n$ given by 
    \[
    (\alpha:\beta:1)\mapsto (\alpha\beta,\; \alpha^n-\beta^n).
    \]
\end{lem}

\begin{proof}
    Suppose $(\alpha:\beta:1)\in F_{n}(K)$. 
    Observe the identity given by
    \[
        1 - (\alpha^n - \beta^n)^2 
        = 
        (\alpha^n + \beta^n)^2 - 
        (\alpha^n - \beta^n)^2 
        = 4(\alpha\beta)^n.
    \]
    In particular
    \[
    (\alpha^n-\beta^n)^2 = -4(\alpha\beta)^n +1.
    \]
\end{proof}

We now determine $C_n(\Q)$ for $n=6, 7$ and $8$ as we shall require these computations in the proofs 
of the main theorems. For $n=6$ and $8$, the hyperelliptic curve $C_n$ admits a degree $2$ map $\pi$ 
to an elliptic curve $E$ of rank $0$ over $\Q$. Looking at the preimages of the points in $E(\Q)$ under 
$\pi$ allows us to determine $C_n(\Q)$ in a straightforward manner. We find that the Jacobian of the hyperelliptic curve 
$C_7$ has rank $0$ over $\Q$, and we are thus able to determine $C_7(\Q)$ by inspecting the Riemann--Roch spaces of degree 
$1$ divisors on $C_7$.

\begin{lem}\label{lem:C6rational}
    $C_6(\Q) = \{(0,1),\, (0,-1)\}$.
\end{lem}

\begin{proof}
    Let $E/\Q$ be the elliptic curve with Cremona label \texttt{27a3}, and consider the Weierstrass model of $E$ given by 
    \[
    E:\; y^2 + y = x^3.
    \]
    There is a degree $2$ map $\pi: C_6\rightarrow E$ given by 
    \[
    \pi(x,y) = (-x^2,\; (y-1)/2).
    \]
    Suppose $P\in C_6(\Q)$ then $\pi(P)\in E(\Q)$.
    Using \texttt{Magma}, we find that 
    \[
    E(\Q) = \Z/3\Z\cdot (0,0).
    \]
    By looking at the preimage of the points in $E(\Q)$ under $\pi$, we conclude that 
    \[
    P = (0,1) \qquad \text{ or } \qquad P= (0,-1).
    \]
\end{proof}

\begin{lem}\label{lem:C7rational}
    $C_{7}(\Q)=\{(0,1), (0,-1), \infty\}$.
\end{lem}

\begin{proof}
    We first determine $J(\Q)$, where $J$ is the Jacobian of $C_7$. Using \texttt{Magma}, we find that $J$ has rank $0$ over $\Q$ (using the \texttt{RankBounds} command). To compute the torsion subgroup of $J(\Q)$, we apply the \texttt{Magma} implementation of an algorithm due to M\"{u}ller and Reitsma~\cite{Muller}. This algorithm is based on the explicit theory of Kummer varieties associated to genus $3$ hyperelliptic curves due to Stoll~\cite{Stoll17}. The aforementioned algorithm returns that the torsion subgroup of $J(\Q)$ has order $7$. Let 
    \[
    P_0=\infty,\quad P_1=(0,-1),\quad P_2=(0,1)\in C_7(\Q).
    \]
    The point $[P_0-P_1]\in J(\Q)$ has order 7. 
    Therefore
    \[
    J(\Q)=\Z/7\Z\cdot [P_0-P_1].
    \]
    Suppose $P$ is a rational point on $C_7$. 
    Then
    \[
        [P]-[P_1] \sim a[P_0-P_1]
    \]
    for an integer $0\leq a\leq 6$, where $\sim$ denotes linear equivalence on $J$. In other words,
    \[
    [P] = [P_1 + a(P_0-P_1)] + \divv(f)
    \]
    where $f\in L_a=L([P_1+ a(P_0-P_1)])$. We used \texttt{Magma} to compute a basis of the Riemann--Roch space $L_a$ for each value of $a$ in the range given above. 
    \begin{itemize}
    \item For $a\geq 3$, the dimension of $L_a$ equals $0$.
    \item 
    If $a=2$ then 
    \[
    [P] = [2P_0-P_1] + \divv(f)
    \]
    where $f\in L(2P_0-P_1)$. A basis of $L(2P_0-P_1)$ is given by $x$. Note that $\divv(x)=P_1+P_2-2P_0$. Thus in this case $P=P_2$.
    \item 
    If $a=1$ then 
    \[
    [P]=[P_0]+\divv(f)
    \]
    where $f\in L(P_0)$. A basis of $L(P_0)$ is given by $1$. Thus $P=P_0$.
    \item If $a=0$ then
    \[
    [P]=[P_1] +\divv(f)
    \]
    where $f\in L(P_1)$.
    A basis of $L(P_1)$ is given by $1$. Thus $P=P_1$.
    \end{itemize}
\end{proof}

\begin{lem}\label{lem:C8rational}
    $C_8(\Q) = \{(0,1),\, (0,-1)\}$.
\end{lem}

\begin{proof}
Let $E/\Q$ be the elliptic curve with Cremona label \texttt{64a4}, and consider the Weierstrass model of $E$ given by 
\[
E:\; y^2 = x^3 + x.
\]
There is a degree $2$ map $\pi: C_8\rightarrow  E$
given by 
\[
\pi(x,y) = \left(\dfrac{1-y}{2x^4},\;\dfrac{y-1}{2x^6}\right).
\]
If $P\in C_8(\Q)$ then $\pi(P)\in E(\Q)$.
We find using \texttt{Magma} that
\[
E(\Q) = (0,0)\cdot \Z/2\Z.
\]
It immediately follows that 
\[
C_8(\Q)=\{(0,1),\; (0,-1)\}.
\]
\end{proof}

We also make repeated use of the following observation throughout the paper.

\begin{lem}
    \label{lem:mapequal}
    Let $C, C'$ be curves defined over $\Q$. 
    Let $m\geq 2$ be an integer. 
    Suppose there is a degree $m$ map $\pi:C\rightarrow C'$ defined over $\Q$. Let $K$ be a primitive number field of degree $d > m$ such that $C'(K)=C'(\Q)$. 
    Then $C(K)=C(\Q)$.
\end{lem}

\begin{proof}
    Let $P\in C(K)$. Under the assumptions of the theorem, $\Q(\pi(P))=\Q$. Thus 
    \[
    \deg(\Q(P))\leq \deg(\pi)=m < d.
    \]
    Note that $\Q(P)\subset K$. 
    It immediately follows that $\Q(P)=\Q$, since $K$ is a primitive number field. 
\end{proof}

\begin{remark}
In our applications of Lemma $6$, the curve $C'/\Q$ is an elliptic curve. In this case, for a given number field $K$, the implementation of certain descent algorithms in \texttt{Magma} \cite{Magma} means the determination of $C'(K)$ is often a reasonable task.
\end{remark}

\section{Proof of Theorem~\ref{thm:F8quartic}: $n = 7$}\label{sec:F87}

Let $K$ be a quartic number field such that $\Gal(\tilde{K}/\Q)=A_{4}$, where $\tilde{K}$ denotes the Galois closure of $K$. Suppose $P\in F_{7}(K)$ is a non-trivial point. 
Let $\pi$ be the morphism defined in Lemma~\ref{lem:hyperellipticlem}. 
Then $\pi(P)=Q$ where $Q\in C_7(K)$.
Let
\[
P_0=\infty, \quad P_1=(0,-1), \quad P_2=(0,1) \in C_7(\Q). 
\]
Write $J$ for the Jacobian of $C_7$. 
Write $D$ for the effective degree $4$ divisor obtained from taking the sum of the Galois conjugates of $Q$. Then $D-4P_0\in J(\Q)$. 
Recall from the proof of Lemma~\ref{lem:C7rational} that
\[
    J(\Q)=\Z/7\Z\cdot [P_1-P_0].
\]
Therefore
\[
D-4P_0\sim a(P_1-P_0)
\]
for some integer $0\leq a\leq 6$, where $\,\sim\,$ denotes linear equivalence on $J$. 
Therefore 
\[
D = 4P_0 + a(P_1-P_0) + \divv(f)
\]
where $f\in L_a=L(4P_0 + a(P_1-P_0))$. We note that the hyperelliptic divisor is $2P_0$.
Moreover, the divisor of the differential $dx/y$ is $4P_0$. 
It follows from Corollary~\ref{cor:degree4} that the Riemann--Roch space $L_a$ has dimension $2$ for $a\neq 0$ and dimension $3$ for $a=0$. First suppose $a=0$. 
Then 
\[
D = 4P_0 + \divv(f)
\]
where $f\in L(4P_0)$. A basis of $L(4P_0)$ can be computed in \texttt{Magma}, and one is given by 
\[
1,\; x,\; x^2.
\]
In this case, the $x$-coordinate of $Q\in C_7(K)$ is the root of a quadratic polynomial. Write $Q=(a,b)\in C_7(K)$. There are three cases to consider.
\begin{itemize}
    \item \textbf{Case 1: $a,b\in \Q$.} In this case, by Lemma~\ref{lem:C7rational},
\[
Q=(a,b)\in C_{7}(\Q)=\{ \infty, \; (0,-1),\; (0,1) \}.
\]
In each case, $Q$ corresponds to a trivial point $P\in F_7(K)$.
    \item \textbf{Case 2:} $a\in \Q$, $b^2\in L/\Q$, where $[L:\Q]=2$.
In this case, $Q$ is defined over a quadratic field. 
Recall that $K$ is a primitive number field. 
Thus $L\not\subseteq K$ and $Q\not\in C_7(K)$. 

\item \textbf{Case 3:} $a\in L/\Q$, where $[L:\Q]=2$.
Again, since $K$ is a primitive number field, $Q\not\in C_7(K)$.
\end{itemize}

It remains to consider the cases $a\neq 0$. 
Write $g, h$ for a basis of $L_a$. 
Then 
\[
f = g - uh
\]
for some $u\in\Q$. 
For $a=1, 2, 5$ and $6$, the linear system has fixed base points and therefore the divisor $D$ is reducible. 
For example, when $a=1$ a basis for $L_a$ is given by $g=x$ and $h=1$. Thus,
\[
D+\divv(x-u)=3P_0+P_1+\divv(x-u)=P_0+P_1+(u,v)+(u,v^\prime)
\]
where $v=\sqrt{-4u^7+1}$ and $v^\prime=-\sqrt{-4u^7+1}$.
For all such $a$, we conclude $Q\not\in C_7(K)$ as $D$ is irreducible.  
It remains to consider $a=3$ and $a=4$. 
First suppose $a =3$.
A basis of $L_a$ is given by 
\[
g\coloneqq\frac{y-1}{x^3},\quad  h\coloneqq1.
\]
Then
\[
D = 4P_0 + 3(P_1-P_0) + \divv(g-u)
\]
for some $u\in\Q$. 
We can view $g$ as a degree $4$ morphism $C_7\rightarrow \PP^1$, and through this lense $D=g^{-1}(u)$.
Let $r=x,\; s=\frac{y-1}{x^3}$. 
We obtain a plane model for $C_7$ given by 
\[
C(r,s):\;4r^4+s^2r^3+2s=0.
\]
With respect to the plane model $C$, the degree $4$ morphism $g: C\rightarrow \PP^1$ is $g(r,s)=s$.
The discriminant of $C$ with respect to $r$ is
\[
\Delta_{r}(C)=(2s)^2\cdot (-3^3s^8+2^{15}s).
\]
Note that $\Delta(\mathcal{O}_{K})$ is a rational square, where $\mathcal{O}_{K}$ denotes the ring of integers of $K$, since $\Gal(\tilde{K}/\Q)=A_{4}$. 
It follows that $\Delta_{r}(C)$ is a rational square.
Write
\[
\frac{\Delta_r(C)}{(2s^5)^2}=\frac{2^{15}}{s^7}-3^3.
\]
It immediately follows that
\[
\left(\frac{2^2}{s},\, \frac{\sqrt{\Delta_r(C)}}{2s^5}\right)\in H(\Q)
\]
where $H/\Q$ is the hyperelliptic curve defined by
\[
H \; :\; y^2=2 x^7-3^3.
\]
It suffices to prove $H(\Q)=\{\infty\}$, since $\infty\in H(\Q)$ does not correspond to a quartic point $(r,s)$ on $C$. Write $J_{H}$ for the Jacobian of $H$. 
Using \texttt{Magma}, we find that $J_{H}$ has rank $0$ over $\Q$ (using the \texttt{RankBounds} command), and that $J_{H}$ has good reduction away from the primes $2, 3$ and $7$. Recall that, for a prime $p$ of good reduction, the torsion subgroup of $J_H(\Q)$ injects into $\tilde{J}(\FF_p)$ (see, e.g., \cite[Appendix]{Katz}). 
By computing the greatest common divisor of $\tilde{J}(\FF_p)$ for the first $11$ primes of good reduction for $J$, we find that $J_{H}(\Q)$ is trivial. This proves that $H(\Q)=\{\infty\}$.  

It remains to consider the case $a=4$. In this case, a basis of $L_a$ is given by 
\[
    m\coloneqq \frac{y-1}{x^4},\qquad n\coloneqq 1.
\]
Similarly to the previous case, we obtain a plane model for $C_7$ given by 
\[
X(v,w):\; w^2v^4+4v^3-2w = 0,
\]
where $v=x, \; w = (y-1)/x^4$ and the degree $4$ morphism $X\rightarrow \PP^1$ is given by $m(v,w)=w$.
The discriminant of $X$ with respect to $v$ is 
\[
\Delta_v(X)=2^{10}w^{2}(-2w^7-27),
\]
which we can rewrite as 
\[
\frac{\Delta_v(X)}{2^{10}w^2} = -2w^7-27.
\]
It follows that $(-w, \sqrt{\Delta_w(X)}/2^{5}w)\in H(\Q)$, where $H$ is the hyperelliptic curve defined above.
Since $H(\Q)=\{\infty\}$ this completes the proof.

\section{Proof of Theorem~\ref{thm:F8quartic}: $n=8$}\label{sec:F88}

Let $K$ be a quartic number field such that $\Gal(\tilde{K}/\Q)=A_{4}$, where $\tilde{K}$ denotes the Galois closure of $K$. Suppose $P\in F_{8}(K)$ is a non-trivial point. 
Let $\pi$ be the morphism defined in Lemma~\ref{lem:hyperellipticlem}. 
Then $\pi(P)=Q$ where $Q\in C_8(K)$.
Let 
\[
P_0=(0,1),\quad P_1=(0,-1)\in C_8(\Q).
\]
Write $J$ for the Jacobian of $C_8$. Let $D$ be the effective degree $4$ divisor obtained from taking the sum of the Galois conjugates of $\Q$. 
Then $D-2P_0-2P_1\in J(\Q)$. 
Using \texttt{Magma}, we find that $J$ has rank $0$ over $\Q$. By applying the aforementioned algorithm of M\"{u}ller and Reitsma \cite{Muller}, we find that the torsion subgroup of $J(\Q)$ is of order $4$. 
The point $[P_1-P_0]$ has order $4$ and thus
\[
J(\Q)=\Z/4\Z\cdot [P_1-P_0].
\]
Therefore
\[
D-2P_0-2P_1\sim a(P_1-P_0)
\]
for some integer $0\leq a\leq 3$.
Therefore 
\[
D = 2P_0+2P_1 + a(P_1-P_0) + \divv(f)
\]
where $f\in L_a=L(2P_0+2P_1 + a(P_1-P_0))$. 
By Corollary~\ref{cor:degree4}, the Riemann--Roch space $L_a$ has dimension $3$ for $a=0$ and dimension $2$ otherwise. For $a=0$, using \texttt{Magma}, we see that a basis of $L_a$ is given by $1, 1/x, 1/x^2$. Similarly to the case of $n=7$, we conclude that $Q\in C_8(\Q)$. 
By Lemma \ref{lem:C8rational},
\[
C_8(\Q)=\{P_0, \;P_1\}.
\]
In this case, we conclude that $P\in F_8(K)$ is a trivial point.

Next suppose $a=1$. Then $D=P_0+3P_1+\divv(f)=2P_1+P_0+P_1+\divv(f)$. Note that $P_0+P_1$
is a hyperelliptic divisor. Thus the linear system is a $g_4^1$ containing a $g_2^1$.
It follows that $D$ is reducible, giving a contradiction. The case $a=3$ is similar.

It remains to consider the case $a=2$. 
Using \texttt{Magma}, we find that a basis of $L_a$ 
is given by 
\[
g\coloneqq \frac{y+1}{x^4},\qquad h:=1 
\]
and so
\[
D = 4P_1 + \divv(g-u)
\]
for some $u\in \Q$. 
As in the proof of the case $n=7$, we view $g$ as a degree $4$ morphism $C_8\rightarrow \PP^1$ and thus $D=g^{-1}(u)$.  
Let $r=x, \,s=(y+1)/x^4$. 
A plane model of $C_8$ is then given by 
\[
C(r,s): \; 4r^4 + r^4s^2 - 2s = 0
\]
and the degree $4$ morphism $g:C\rightarrow \PP^1$ becomes $g(r,s) = s$.
The discriminant of $C(r,s)$ with respect to $r$ is
\[
\Delta_r(C)=-2^{11}s^{3}(s^2+2^2)^3.
\]
Recall that $\Delta(\mathcal{O}_K)$ is a rational square where $\mathcal{O}_{K}$ denotes the ring of integers of $K$, since $\Gal(\tilde{K}/\Q)=A_{4}$.
It follows that $\Delta_r(C)$ is a rational square. 
In other words, $(s,\, \sqrt{\Delta_r(C)}/2^{5}s(s^2+4))\in H(\Q)$ where $H$ is the curve defined by
\[
H:\; y^2=-2x^3-8x.
\]
Moreover $H$ is isomorphic to the elliptic 
curve $E/\Q$ with Cremona label \texttt{64a4} given by the Weierstrass equation
\[
E:\; y^2 = x^3 + 16x.
\]
Using \texttt{Magma}, we find that 
\[
E(\Q) \cong\Z/2\Z.
\]
Thus
\[
H(\Q)=\{\infty, \;(0,0)\}.
\]
It is clear that neither of the rational points on $H$ correspond to a quartic point $(r,s)$ on the plane model $C$. This completes the proof.
\section{Proof of Theorem~\ref{thm:F6}}\label{sec:proofF6}

    Suppose $P=(\alpha:\beta:1)\in F_{6}(K)$. 
   Let
   \[
    C_{6}:\; y^2=-4x^6+1.
    \]
Let
\[
a=\alpha\beta, \quad b=\alpha^6-\beta^6.
\]
Then $Q=(a,b)\in C_{6}(K)$ by Lemma~\ref{lem:hyperellipticlem}. 
We show that $Q$ corresponds to a trivial point $P\in F_{6}(K)$.
Consider the Weierstrass model of the elliptic curve $E_{27}/\Q$ with Cremona label \texttt{27a3} given by 
\[
E_{27}\;:\; y^2 + y = x^3.
\]
In particular, there is a degree $2$ map from $C_{6}$ to the elliptic curve $E_{27}$ given by 
\[
\pi: C_{6}\rightarrow E_{27},\qquad (x,y)\mapsto \left(-x^2,\;\frac{1}{2}(y-1)\right).
\]
Consider now the Weierstrass model of the elliptic curve $E_{432}/\Q$ with Cremona label \texttt{432b1} given by 
\[
E_{432}\;:\; y^2 = x^3 - 4.
\]
There is a degree $2$ map from $C_{6}$ to the elliptic curve $E_{432}$ given by 
\[
C_{6}\rightarrow E_{432},\qquad (x,y)\mapsto \left(\dfrac{1}{x^2},\;\dfrac{y}{x^3}\right).
\]
Let $E/\Q$ be one of the elliptic curves $E_{27}$ or $E_{432}$ with $E(K)=E(\Q)$. Then $Q\in C_{6}(\Q)$ by Lemma~\ref{lem:mapequal}. By Lemma~\ref{lem:C6rational}, 
\[
C_{6}(\Q)=\{(0,1),\;(0,-1)\}.
\]
Since $Q=(a,b)\in C_{6}(\Q)$, this implies that $a=0$ and $\alpha\beta=0$, i.e., $P\in F_6(K)$ is a trivial point.

\section{Proof of Theorem~\ref{thm:F8}}\label{sec:proofF8}

    Suppose $P=(\alpha:\beta:1)\in F_{8}(K)$. 
   Let
   \[
    C_{8}:\; y^2=-4x^8+1.
    \]
Let
\[
a=\alpha\beta, \quad b=\alpha^8-\beta^8.
\]
Then $Q=(a,b)\in C_{8}(K)$ by Lemma~\ref{lem:hyperellipticlem}. 
We work with the Weierstrass model of the elliptic curve $E/\Q$ with Cremona label \texttt{64a4} given by 
\[
    E\;:\; y^2=x^3+x.
\]
Recall there is a degree $2$ map $\pi$ from $C_{8}$ to the elliptic curve $E$ given by 
\[
\pi: C_{8}\rightarrow E,\qquad (x,y)\mapsto \left(\dfrac{1-y}{2x^4},\;\dfrac{y-1}{2x^6}\right).
\]
Recall that $E(K)=E(\Q)$ by assumption, and therefore $Q\in C_{8}(\Q)$ by Lemma~\ref{lem:mapequal}. 
By Lemma~\ref{lem:C8rational},
\[
C_{8}(\Q)=\{(0,1),\; (0,-1)\}.
\]
Thus if $Q=(a,b)\in C_{8}(K)=C_{8}(\Q)$ then $a=0$, and moreover $\alpha\beta=0$, i.e., $P\in F_8(K)$ is a trivial point.\\

All supporting \texttt{Magma} computations 
are available on \texttt{GitHub} at 
\[\text{\url{https://github.com/MaleehaKhawaja/PrimitivePointsFermat}}
\]

\bibliographystyle{abbrv}
\bibliography{PrimitiveFermat.bib}

\begin{thebibliography}{10}

\bibitem{Aigner34}
A.~Aigner.
\newblock {\"U}ber die {M{\"o}glichkeit} von {{\(x^4 + y^4 = z^4\)}} in
  quadratischen {K{\"o}rpern}.
\newblock {\em Jahresber. Dtsch. Math.-Ver.}, 43:226--228, 1934.

\bibitem{Magma}
W.~Bosma, J.~Cannon, and C.~Playoust.
\newblock The {M}agma algebra system. {I}. {T}he user language.
\newblock {\em J. Symbolic Comput.}, 24(3-4):235--265, 1997.
\newblock Computational algebra and number theory (London, 1993).

\bibitem{BremnerChoudhry}
A.~Bremner and A.~Choudhry.
\newblock The {F}ermat cubic and quartic curves over cyclic fields.
\newblock {\em Period. Math. Hungar.}, 80(2):147--157, 2020.

\bibitem{Rohrlich1978}
B.~H. Gross and D.~E. Rohrlich.
\newblock Some results on the {M}ordell-{W}eil group of the {J}acobian of the
  {F}ermat curve.
\newblock {\em Invent. Math.}, 44(3):201--224, 1978.

\bibitem{Hartshorne}
R.~Hartshorne.
\newblock {\em Algebraic geometry}.
\newblock Graduate Texts in Mathematics, No. 52. Springer-Verlag, New
  York-Heidelberg, 1977.

\bibitem{JarvisMeekin}
F.~Jarvis and P.~Meekin.
\newblock The {F}ermat equation over {${\Bbb Q}(\sqrt{2})$}.
\newblock {\em J. Number Theory}, 109(1):182--196, 2004.

\bibitem{Katz}
N.~M. Katz.
\newblock Galois properties of torsion points on abelian varieties.
\newblock {\em Invent. Math.}, 62(3):481--502, 1981.

\bibitem{KhawajaJarvis}
M.~Khawaja and F.~Jarvis.
\newblock Fermat's last theorem over {$\Bbb Q(\sqrt 2,\sqrt 3)$}.
\newblock {\em Algebra Number Theory}, 19(3):457--480, 2025.

\bibitem{khawaja2023primitive}
M.~Khawaja and S.~Siksek.
\newblock Primitive algebraic points on curves.
\newblock {\em Res. Number Theory}, 10(3):Paper No. 57, 20, 2024.

\bibitem{MatthewKlassen1997}
M.~Klassen and P.~Tzermias.
\newblock Algebraic points of low degree on the {F}ermat quintic.
\newblock {\em Acta Arith.}, 82(4):393--401, 1997.

\bibitem{Kraus18}
A.~Kraus.
\newblock Quartic points on the {F}ermat quintic.
\newblock {\em Ann. Math. Blaise Pascal}, 25(1):199--205, 2018.

\bibitem{Mo}
L.~J. Mordell.
\newblock The {D}iophantine equation {$x\sp{4}+y\sp{4}=1$} in algebraic number
  fields.
\newblock {\em Acta Arith.}, 14:347--355, 1967/68.

\bibitem{Muller}
J.~S. M\"uller and B.~Reitsma.
\newblock Computing torsion subgroups of {J}acobians of hyperelliptic curves of
  genus 3.
\newblock {\em Res. Number Theory}, 9(2):Paper No. 23, 26, 2023.

\bibitem{SilvermanAEC}
J.~H. Silverman.
\newblock {\em The arithmetic of elliptic curves}, volume 106 of {\em Graduate
  Texts in Mathematics}.
\newblock Springer-Verlag, New York, 1986.

\bibitem{Stoll17}
M.~Stoll.
\newblock An explicit theory of heights for hyperelliptic {J}acobians of genus
  three.
\newblock In {\em Algorithmic and experimental methods in algebra, geometry,
  and number theory}, pages 665--715. Springer, Cham, 2017.

\end{thebibliography}

\end{document}